\documentclass[12pt]{amsart}

\usepackage[utf8]{inputenc}
\usepackage{listings}
\usepackage{graphicx}
\usepackage{pgffor} 
\usepackage{float}
\usepackage{graphics} 
\usepackage{fancyhdr}
\usepackage{amsfonts}
\usepackage{amsmath}
\usepackage{amssymb,amsfonts}
\usepackage{enumerate}
\usepackage{graphicx}
\usepackage{amsmath,amssymb,amsthm}
\usepackage{enumerate}
\usepackage{enumitem}
\usepackage{setspace}
\usepackage{amsthm}
\usepackage{mathrsfs}
\usepackage{nccmath}
\usepackage{mathtools}
\usepackage{appendix}
\usepackage{subcaption}
\usepackage{stmaryrd}
\usepackage{tikz-cd}
\usepackage[OT2,T1]{fontenc}
\usepackage[normalem]{ulem}

\graphicspath{{Images/}}

\newcommand{\C}{\mathbb{C}}
\newcommand{\Z}{\mathbb{Z}}
\newcommand{\Q}{\mathbb{Q}}

\newcommand{\G}{\mathbb{G}}

\newcommand{\Br}{\operatorname{Br}}
\newcommand{\m}[2]{\mu_{#1}^{\otimes #2}}
\DeclareFontFamily{U}{wncy}{}
\DeclareFontShape{U}{wncy}{m}{n}{<->wncyr10}{}
\DeclareSymbolFont{mcy}{U}{wncy}{m}{n}
\DeclareMathSymbol{\Sha}{\mathord}{mcy}{"58} 

\newtheorem{thm}{Theorem}
\newtheorem*{thm*}{Theorem}
\newtheorem{cor}[thm]{Corollary}
\newtheorem{lem}[thm]{Lemma}
\newtheorem{prop}[thm]{Proposition}

\newtheorem*{ack*}{Acknowledgements}
\theoremstyle{remark}

\let\oldproofname=\proofname
\renewcommand{\proofname}{\rm\bf{\oldproofname}}

\DeclareSymbolFont{cyrletters}{OT2}{wncyr}{m}{n}
\DeclareMathSymbol{\Sha}{\mathalpha}{cyrletters}{"58}

\title{Degree $2$ Del Pezzo Surface Bundles And Stable Rationality}
\author{Wenhao Li}

\address{Courant Institute of Mathematical Sciences, New York University, New York, 10012, U.S.A.}
\email{wl1693@nyu.edu}

\usepackage{filecontents}
\usepackage[style=alphabetic,backend=biber]{biblatex}
\renewbibmacro{in:}{}
\renewbibmacro*{volume+number+eid}{%
\printfield{volume}%
\setunit{, no.}%
\printfield{number}%
\setunit{\bibeidpunct}%
\printfield{eid}}
\begin{document}

\begin{abstract}

We study the arithmetic
of del Pezzo surfaces $Y$ of degree 2 over a function field, and in particular, the cokernel of the homomorphism from the Picard group to the Galois-invariants of the geometric Picard group
$\operatorname{Pic} Y \rightarrow(\operatorname{Pic} \bar{Y})^{G}$. Applying this to a fibration $\pi:X\to S$ in del Pezzo surfaces of degree 2 over a rational surface $S$, we construct examples with nontrivial  relative unramified cohomology group $H^2_{nr,\pi}(k(X)/k)$. A specialization argument implies the failure of stable rationality of varieties specializing to $X$. 
\end{abstract}

\maketitle

\section{Introduction}

    Recently, there has been dramatic progress in the study of rationality properties of 
  higher-dimensional varieties over uncountable algebraically closed fields.  Specialization arguments allowed to show the failure of stable rationality  
for several important classes of smooth projective varieties. These arguments are based on a careful degeneration of the variety of interest to a (mildly singular) {\it reference} variety $X$.
Suitable (stable) birational invariants of $X$ can be used to show the 
    failure of stable
    rationality of the original variety (examples include 
     \cite{CTP16A}, \cite{CTP16B}, \cite{HKT16}, \cite{T16}, \cite{HPT18}, \cite{ABP18}, \cite{S18}, \cite{S19A}, \cite{S19B}, \cite{O19}, \cite{KO20}, \cite{NO22},  \cite{P23}, and others). 
   In particular, recent results include proofs of failure of stable rationality of very general hypersurfaces  in $\mathbb{P}^2\times\mathbb{P}^3$ of bidegree (2,2), (2,3), (3,3) and ($d$,2) for $d\geq3$. Work of Ahmadinezhad and Okada in \cite{AO16}, along with \cite{HPT18} and \cite{ABP18}, shows that a very general hypersurface of bidegree (2,$d$) in $\mathbb{P}^2\times\mathbb{P}^3$ is not stably rational.

    Here, we are interested in reference varieties of dimension $4$. Papers of Hassett, Pirutka, and Tschinkel  \cite{HPT18} and \cite{P23} contain  examples of reference varieties, e.g., fibrations in quadric or cubic surfaces over $\mathbb{P}^2$. Here we extend these results to another class of fibrations in (geometrically) rational surfaces over $\mathbb{P}^2$: fibrations in del Pezzo surfaces of degree $2$. We construct a new reference variety of this type, see  Proposition \ref{refv}. 
    
    Throughout this note, $k$ is an uncountable algebraically closed field of characteristic not equal to 2.   

The birational invariants we study are the relative unramified cohomology groups $H^i_{nr,\pi}(k(X)/k)$. These are  certain subgroups of the unramified cohomology groups $H^i_{nr}(k(X)/k)$, determined by 
a fibration $\pi:X\to S$. 
The general method for computing unramified cohomology for fibrations goes back to Colliot-Th\'{e}l\`{e}ne and Ojanguren in \cite{CTO89}, who generalized the example of Artin and Mumford in \cite{AM72}. 

The specialization method was  introduced by Voisin in \cite{V15} and developed by Colliot-Th\'{e}l\`{e}ne and Pirutka in \cite{CTP16A}.  
We will use a refinement by Schreieder \cite{S19A} which gives  more freedom in constructing examples, without analysis of singularities of the reference variety, at the cost of additional restrictions for unramified classes; these restrictions will be satisfied in our case.
    
    Using the reference variety we construct, a direct application of the specialization method yields the following result (see Corollary \ref{coromain}), which was one of the cases obtained by Krylov and Okada in \cite{KO20} using a different proof:

    \begin{thm} 
   Let $k$ be an uncountable algebraically closed field of char $(k) \neq 2$. A double cover of $\mathbb{P}^2\times \mathbb{P}^2$ branched along a very general divisor of bidegree $(2q,4)$is not stably rational for $q\geq 4$.   
    \end{thm}
    
\subsection*{Acknowledgements}
    The author was partially supported by NSF grant DMS-2201195. The Author is very grateful of Alena Pirutka's thorough guidance throughout this project, and would also like to thank Asher Auel, Fedor Bogomolov, Anthony V\'{a}rilly-Alvarado, and Zhijia Zhang for helpful discussions and suggestions.

\section{Relative unramified cohomology groups}

Let $K=k(X)$ be the function field of an integral algebraic variety $X$ over $k$, $n>1$ an integer invertible in $K$, and $\mu_n$ the \'{e}tale sheaf defined by the $X$-group scheme of $n$-th roots of unity. When $j$ is a positive integer, $\mu_n^{\otimes j}$ is the tensor product of $j$ copies of $\mu_{n}$. If $j=0$, $\mu_n^{\otimes 0}=\Z/n$, and if $j<0$ then $\mu_n^{\otimes j}=\operatorname{Hom}\left(\mu_n^{\otimes -j},\Z/n\right)$. Denote by $H^i(K,\mu_n^{\otimes j})$ the $i$th Galois cohomology group.  In particular,  $H^2(K,\mu_n)=\Br K[n]$.

In this note we will be mostly interested in the case $i=2$ and $n=2$ and we will write $H^2(K,\Z/2\Z)$ for the corresponding Galois cohomology group.

Consider a divisorial discrete valuation $v$ on $K$, trivial on $k$. Let  $$R\coloneqq \left\{ x\in K|v(x)\geq0\right\}$$ be the valuation ring, and let $\kappa(v)$ be the residue field. There is a residue homomorphism

$$
\partial_v^i: H^i(K,\mu_n^{\otimes j})\rightarrow H^{i-1}(\kappa(v),\mu_n^{\otimes(j-1)}),
$$
which factors through the completion $K_v$ of $K$ at $v$. In the case $i=2$ and $n=2$  we have
\begin{equation}\label{res2}
\partial_v^2(a, b)=(-1)^{v(a) v(b)} \overline{\frac{a^{v(b)}}{b^{v(a)}}},
\end{equation}
where $(a, b):=a \cup b \in H^2\left(K, \Z/2\Z\right)$. 

Here $\overline{\frac{a^{v(b)}}{b^{v(a)}}}$ denotes the image of $\frac{a^{v(b)}}{b^{v(a)}}$ in $\kappa(v)^\times/\kappa(v)^{\times 2}$. 
See \cite{CT92}, \cite{CTS21} for  more details on the residue homomorphism.

In particular, there is an identification \cite[$\S 3.3$]{CT92}:
$$
\operatorname{Ker}(\partial_v^i)=\operatorname{Im}\left[H^i_{\acute{e}t}(R,\mu_n^{\otimes j})\rightarrow H^i(K,\mu_n^{\otimes j})\ \right].
$$
Define the $i$-th unramified cohomology of $K$ over $k$ to be the intersection of kernels of residue maps of all 
divisorial  discrete valuations $v$ of $K=k(X)$
trivial on $k$ \cite{CT92}:
$$
H^i_{nr}(k(X)/k,\mu_n^{\otimes j})\coloneqq\bigcap_v\operatorname{Ker}(\partial^i_v).$$

Recall that the Brauer group of an integral algebraic variety $X$ is 
$$
\Br X\coloneqq H^2_{et}(X,\G_m).
$$
If $X$ is smooth, $\Br X$ maps injectively into $\Br k(X)$ \cite{SGA5}. If $X$ is smooth and projective, and $n$ is invertible in $k$, then 
$$
\Br X[n] \cong H^2_{nr}(k(X)/k,\mu_n),
$$ 
see \cite{CT92}.

Let $S$ be a smooth projective integral variety over $k$, and put $L=k(S)$. Let $X$ be an integral projective variety with a dominant morphism $\pi: X\rightarrow S$, and $X_L$ or $X_{k(S)}$ the generic fiber of $\pi$. Define the $n$-th relative unramified cohomology group:
\begin{multline*}
H^i_{nr,\pi}(k(X)/k,\mu_n^{\otimes j})\coloneqq
\operatorname{Im}[H^i(k(S),\m{n}{j})\rightarrow H^i(k(X),\m{n}{j})]\bigcap\\ \cap_x \operatorname{Ker}[H^i(k(X),\m{n}{j})\rightarrow H^i(L_x(X_{L}),\m{n}{j})],
\end{multline*}
where $x$ runs over all scheme points of $S$ of positive codimension, and $L_x$ is the field of fractions of the completed local ring $\hat{\mathcal O}_{S,x}$ \cite{P23}.
In particular:

\begin{prop}\label{relnr} \cite{P23}
One has
    $$H^i_{nr,\pi}(k(X)/k,\mu_n^{\otimes j})\subset H^i_{nr}(k(X)/k,\mu_n^{\otimes j}).$$
\end{prop}

The proof uses the following commutative diagram
\begin{equation}\label{dP23}
 \begin{tikzcd}
L(X_L) \arrow[r,hook] & L_{x_{v}}(X_L)\arrow[r,hook] &L(X_L)_v=k(X)_v \\
L \arrow[r,hook]\arrow[u,hook]& L_{x_{v}}\arrow[u,hook]
\end{tikzcd}   
\end{equation}
and the factorization 
\begin{equation}\label{resfac}
\partial_v^i: H^i\left(k(X), \mu_n^{\otimes j}\right) \rightarrow H^i\left(k(X)_v, \mu_n^{\otimes j}\right) \rightarrow H^{i-1}\left(\kappa(v), \mu_n^{\otimes(j-1)}\right).
\end{equation}
We will return to (\ref{dP23}) below.

\section{Example of a Reference Variety}

\subsection{Notation}

Let $S$ be a smooth projective rational surface over $k$.

Let $L=k(S)$ be the function field; in our example, $S=\mathbb{P}^2$ and $k$ can be $\C$, so that $L=\C(x,y)$. Let $X$ be a smooth projective variety fibered over $S$ with generic fiber $X_L$ a smooth del Pezzo surface of degree 2 over $L$. Note that $L(X_L)=k(X)$.

Let $v$ be a valuation on $L(X_L)=k(X)$, trivial on $k$. Denote by $x_v$ the center of $v$ on $S$: it is the image of the closed point of $\operatorname{Spec}R$ in $S$, where $R$ is the valuation ring. Let $k(X)_v$ be the completion, and $L_{x_v}$  the field of fractions of the completed local ring $\hat{O}_{S,x_v}$.

Following \cite{CTO89}, \cite{P23}, we are interested in nonzero elements in $$H^2_{nr,\pi}(k(X)/k,\Z/2\Z).$$
By Proposition \ref{relnr}, this is a stronger condition than requiring that the elements are unramified. This is a sufficient condition in order to use the refinement of the specialization method by Schreieder, see \cite{P23}.

\subsection{Key properties} \cite{CTO89}, \cite{P23} 

To construct elements in relative unramified cohomology groups, 
we look at $\alpha\in H^2(L,\Z/2\Z)=\operatorname{Br}L[2]$ and its image $\alpha^\prime\in H^2(L(X_L),\Z/2\Z)$.
We need to compute: 
\begin{itemize}
    \item $\operatorname{Ker}[H^2(L,\Z/2\Z)\rightarrow H^2(L(X_{L}),\Z/2\Z)]$. This will allow us  to insure that $\alpha'\neq 0$.
    \item $\operatorname{Ker}[H^2(L_{x_v},\Z/2\Z)\rightarrow H^2(L_{x_v}(X_{L}),\Z/2\Z)]$ for $x_v\in S$ a point of positive codimension. This will allow  us to deduce that the image of $\alpha'$ in $H^2(L_{x_v}(X_{L}),\Z/2\Z)$ is zero.
\end{itemize}

\subsection{The reference variety}
In the following propsition, we let $L$ be an arbitrary field of characteristic not equal to 2 and  containing eights roots of unity. Consider the degree 2 del Pezzo surface $Y\subset \mathbb P(1,1,1,2)$ defined by 
$$
Y: \;w^2=Au^4+Bv^4+Ct^4.
$$

\begin{prop}\label{ABCKer}
    If $C=ABd^2$ for some $d\in L^\times/(L^\times)^2$ and if the class  $(A,B)\in\Br L[2]$ is nonzero, then the kernel of the map $$\operatorname{Br} L \rightarrow \operatorname{Br} L(Y)$$ is $\Z/2\Z$, generated by $(A,B)$.
    %2/12 added an assumption that (A,B) is nonzero, otherwise the kernel is not Z/2! 
    If $ABC$ is not a square, the map $$
\operatorname{Br} L \rightarrow \operatorname{Br} L(Y)$$ is injective.
\end{prop}

\begin{proof}
    We first  prove that if $C=ABd^2$, then $(A,B)$ is indeed in the kernel. Note that
    $$
    B=\frac{w^2-Au^4}{v^4+Ad^2t^4}=N_{L(Y)(\sqrt{A})/L(Y)}(\beta),
    $$ 
    where $\beta=\frac{w-\sqrt{A}u^2}{v^2+i\sqrt{A}dt^2}$. If $A$ is a square in $L(Y)$, then $(A,B)=0$ in $L(Y)$. If not, $A$ is a square in $L(Y)(\sqrt{A})$ and $$(A,B)=(A,N_{L(Y)(\sqrt{A})/L(Y)}(\beta))=N_{L(Y)(\sqrt{A})/L(Y)}(A,\beta)=0.$$
    The remaining statements follow from Lemma \ref{Abd2} and Lemma \ref{ABC} below.
\end{proof}

\medskip

\begin{prop}\label{refv}
    Let $k$ be an algebraically closed field of char $(k) \neq 2$.
    %2/12 of characteristic not equal to 2:  same style as in other propositions!
    Let $F\in k[x,y,z]$ be a homogeneous polynomial of degree $2g$,
    a square of a polynomial of degree $g\geq1$. Let $n<2g$ be an even integer,
    $A$ a product of $n$ linear forms $l_1,\ldots,l_n$, and $B$ a product of $2g-n$ linear forms $l_{n+1},\ldots,l_{2g}$, where 
    \begin{itemize}
        \item [(i)] the lines defined by the linear forms
        $$\{l_1,\ldots l_{2g}\}$$ are pairwise distinct;
        \item [(ii)] no three of these lines are concurrent;
        \item [(iii)] the zero locus of $F$ does not contain any of these lines;
        \item [(iv)] $F$ does not vanish at any point of intersection of a linear factor of $A$ and a linear factor of $B$;
        \item [(v)] $AB+F$ is not a square of a homogeneous polynomial in $k[x,y,z]$.
    \end{itemize}
    Let $m$ be an even non-negative integer and 
    $$
    X\subset \mathbb{P}^2_{[x:y:z]}\times \mathbb{P}(1,1,1,2)_{[u:v:t:w]}
    $$ 
    a hypersurface of bidegree $(4g+m,4)$ defined by 
$$
w^2=Az^{4g+m-n}u^4+Bz^{2g+m+n}v^4+ABz^m(AB+F)t^4.
$$
Let $\pi: X\rightarrow \mathbb{P}^2$ be the projection and $L=k(\mathbb{P}^2)$. Then there exists an element $\alpha\in H^2(L,\Z/2\Z)=\operatorname{Br}L[2]$ such that its image $\alpha^\prime\in H^2(k(X),\Z/2\Z)$ is not zero and $\alpha^\prime\in H^2_{nr,\pi}(k(X)/k,\Z/2\Z)$.
\end{prop}

\begin{proof}
We will write $(x,y)$ for the affine coordinates of the chart $z\neq 0$, so that $k(\mathbb{P}^2)=k(x,y)$.
Let 
    $$\alpha=(Az^{4g+m-n},Bz^{2g+m+n})=(A,B)\in H^2(L, \Z/2\Z),
    $$
     where we denote by $A$, $B$ their images in $L$.
    Let us justify that $\alpha'$ is not zero. %Recall that 

    For any of the lines $l_i$ dividing $A$,
    we have a corresponding discrete valuation which we denote by $v_{l_i}$. By formula (\ref{res2}) we have 
    $$\partial^2_{v_{l_i}}(A,B)=(-1)^{1\times0}\overline{\frac{A^{0}}{B^1}}=\overline B
    $$
    in $\kappa(v_{l_i})^*/\kappa(v_{l_i})^{*2}$.  This is not a square because we have no concurrent lines. 
    This shows that $\alpha\neq 0$.
    Since $AB+F$ is not a square by hypothesis, by the non-square case in Proposition \ref{ABCKer}, we deduce that $\alpha^\prime$ is also not zero.

    We now show that $\alpha^\prime$ is in the relative unramified cohomology group. We have the following cases:
    \begin{itemize}
        \item If $x_{v}\in\mathbb{P}_k^2$ is a generic point of any of the lines defined by $l_i$, then $AB+F$ reduces to a  nonzero square in $\kappa(x_v)$, and, by Hensel's lemma, its image in $L_{x_v}$ is a square. Hence the equation is of the form 
    $$
    w^2=Az^{4g+m-n}u^4+Bz^{2g+m+n}v^4+ABz^mt^4
    $$ 
    in $L_{x_v}$, where $Az^{4g+m-n}Bz^{2g+m+n}=ABz^{6g+2m}$. Since $m$ is even, 
    $$
    6k+2m\equiv m \mod 2.
    $$
    A change of notation gives 
    $$
     w^2=A'u^4+B'v^4+A'B'd^2t^4,
    $$
    where $A^\prime=Az^{4g+m-n}$,  $B^\prime=Bz^{2g+m+n}$, and $d$ can be chosen to be $1$ if $6g+m\equiv0\mod4$ and $z$ if $6g+m\equiv2\mod 4$. This is of the form described in Proposition \ref{ABCKer}. Note that $(A',B')=(A,B)$ and hence
    $$
    \alpha^\prime\in \operatorname{Ker}[H^2(L,\Z/2\Z)\rightarrow H^2(L_{x_v}(X_{L_{x_v}}),\Z/2\Z)],
    $$ 
    by Proposition \ref{ABCKer}.

    \item If $x_v$ is a closed point only on $\{A=0\}$ and not on $\{B=0\}$ (similarly, only on $B$, and not on $A$), then $B$ is a nonzero unit in the local ring of $x_v$ and it becomes a square in $\kappa(x_v)=k$ because $k$ is assumed to be algebraically closed, so that the image of $B$ in $L_{x_v}$ is a square by Hensel's lemma. The image of $\alpha$ in $L_{x_v}$ is therefore zero. Using diagram (\ref{dP23}), we see that $\alpha^\prime$ maps to zero in $H^2(L_{x_v}(X_{L_{x_v}}),\Z/2\Z)$.
    
    \item If $x_v$ is at the intersection of lines, then $AB+F$ also reduces to a 
    nonzero square in $\kappa(x_v)$, and we proceed as in the first case.

    \item If $x_v$ is none of the above, then $A$ and $B$ are units in the local ring of $x_v$. This implies that the image of $\alpha$ in $H^2(L_{x_v}, \Z/2\Z)$ is zero    since it then comes from $H^2_{\acute{e}t}(\hat{\mathcal{O}}_{\mathbb{P}^2,x_v},\Z/2\Z)=H^2(\kappa(x_v),\Z/2\Z)=0$. In fact, either $\kappa(x_v)=k$ is algebraically closed or $\kappa(x_v)$ is the function field in one variable over $k$, hence of cohomological dimension $1$.
    
    \end{itemize}
\end{proof}

The following proposition allows us to use $X$ as a ``reference'' variety:

\begin{prop}\label{NSR}
     Let $k$ be an algebraically closed field of char $(k) \neq 2$. Let $S$ be a smooth projective integral variety over $k$. Let $X$ be an integral projective variety and $\pi: X \rightarrow S$ a dominant map with smooth generic fiber. Assume that there is an integer $i>0$ such that
$$
H_{n r, \pi}^i\left(k\left(X\right) / k, \mathbb{Z} / 2 \mathbb{Z}\right) \neq 0 .
$$
    Let $V$ be a proper geometrically integral variety over a field $K$ which degenerates to $X$. Then $V$ is not stably rational over $K$.
\end{prop}
\begin{proof}
    See \cite{P23} and references therein.
\end{proof}

We have the following corollary,
which was proved in \cite{KO20} by Krylov and Okada using a different method:
\begin{cor}\label{coromain}
    Let $k$ be an uncountable algebraically closed field of char $(k) \neq 2$. A double cover of $\mathbb{P}^2\times \mathbb{P}^2$ branched along a very general divisor of bidegree $(2q,4)$ is not stably rational for $q\geq 4$.
\end{cor}
\begin{proof}
Let $V\to \mathbb{P}^2\times \mathbb{P}^2$ be a double cover branched along a very general divisor of bidegree $(2q,4)$, as in the statement. We view it as a fibration in del Pezzo surfaces of degree $2$ by the first projection. 
Now construct $X$ as in Proposition \ref{refv}, with the following: put $g=2$, $n=2$,
and $m=2q-8$. So $A$ and $B$ are quadratic, and the homogeneous polynomial $F$ is of degree $4$. The specific choice of the lines and $F$ can be arbitrary, as long as the conditions in Proposition \ref{refv} are satisfied. With $X$ constructed, we can  degenerate $V$ to it. By Proposition \ref{NSR}, $V$ is not stably rational.
\end{proof}

Now we construct a specific simple example using Proposition \ref{refv} and \ref{NSR} above. Let $X\subset \mathbb{P}^2_{[x:y:z]}\times \mathbb{P}(1,1,1,2)_{[u:v:t:w]}$ defined by the following
$$
w^2=Az^6u^4+Bz^6v^4+AB(AB+F)t^4,
$$
where $A=(x^2+xz+z^2)$, $B=(y^2+yz+z^2)$, and $F=(x+y)^4$.
Let $\pi: X\rightarrow \mathbb{P}^2$ the projection to $\mathbb{P}^2$, $L=k(\mathbb{P}^2)=k(x,y)$, where we denote by $x,y$ again the non-homogeneous coordinates of $L$. As in Corollary \ref{coromain}, we conclude that a very general hypersurface of bidegree $(8,4)$ in $\mathbb{P}^2\times \mathbb{P}(1,1,1,2)$ is not stably rational.

\section{Arithmetics of Diagonal del Pezzo Surfaces}
In this section, we prove the rest of Proposition \ref{ABCKer}. We do this by explicit computation of $(\operatorname{Pic}\bar{Y})^G$, 
following Kresch-Tschinkel \cite{KT04}.

Consider the del Pezzo surface $Y$ of degree $2$ defined by 
\begin{equation}\label{defX}
Y:\; w^2=Au^4+Bv^4+Ct^4.
\end{equation}
This is a double cover of $\mathbb{P}^2$ branched over the quartic $$Au^4+Bv^4+Ct^4=0.$$ The surface $Y$ has 56 
(-1) curves, seen as preimages of bitangent lines of the branch quartic. Denote the 4th root of $A,B,C$ by $a,b,c$ respectively. The (-1) curves are given by the following equations \cite{KT04}:

$$
\begin{array}{rrrr}
\ell_{t, \delta, \pm}: & \delta a u+b v=0, & w= \pm c^2 t^2, \quad \text { with } \delta^4=-1, \\
\ell_{u, \delta, \pm}: & \delta b v+c t=0, & w= \pm a^2 u^2, \quad \text { with } \delta^4=-1, \\
\ell_{v, \delta, \pm}: & \delta c t+a u=0, & w= \pm b^2 v^2, \quad \text { with } \delta^4=-1, \\
\ell_{\alpha, \beta, \gamma}: & \alpha a u+\beta b v+\gamma c t=0, & w=\sqrt{2}(\alpha \beta a b u v+\beta \gamma b c v t+\alpha \gamma a c u t), \\
& & \text { with }(\alpha, \beta, \gamma) \in \mu_4^3 / \mu_2 .
\end{array}
$$
The Picard group $\operatorname{Pic}\bar{Y}$ of $\bar{Y}$
is a free abelian group, with generators  \cite{KT04}:
\begin{equation}\label{basis}
\begin{aligned}
& v_1=\left[\ell_{u, \zeta,+}\right], \quad v_2=\left[\ell_{u, \zeta^3,-}\right], \quad v_3=\left[\ell_{v, \zeta,+}\right], \quad v_4=\left[\ell_{v, \zeta^3,-}\right], \\ & v_5=\left[\ell_{t, \zeta,+}\right], \quad v_6=\left[\ell_{t, \zeta^3,-}\right], \quad v_7=\left[\ell_{i, i, i}\right], \quad v_8=\left[\ell_{t, \zeta^7,-}\right]+\left[\ell_{t, \zeta^3,-}\right]+\left[\ell_{i, i, i}\right] . \\ & 
\end{aligned}
\end{equation}

\subsection{The case $C=ABd^2$}

\begin{lem}\label{Abd2}
    Let $Y$ be the del Pezzo surface over the field $L$ defined by
    $$
    Y:\; w^2=Au^4+Bv^4+Ct^4
    $$
    If $C=ABd^2$ for some $d\in  L^\times/(L^\times)^2$ and if the class  $(A,B)\in\Br L[2]$ is nonzero, then the kernel of $$\operatorname{Br} L \rightarrow \operatorname{Br} L(Y)$$ is $\Z/2\Z$.
\end{lem}
\begin{proof}
    If $C=ABd^2$ for some $d\in L^\times/(L^\times)^2$, $ABC$ is a square in $L$, and the equations simplify to

    $$
\begin{array}{rrrr}
\ell_{t, \delta, \pm}: & \delta a u+b v=0, & w= \pm da^2b^2t^2,\quad \text { with } \delta^4=-1, \\
\ell_{u, \delta, \pm}: & \delta  v+a\sqrt{d} t=0, & w= \pm a^2 u^2, \quad \text { with } \delta^4=-1, \\
\ell_{v, \delta, \pm}: & \delta b \sqrt{d}t+ u=0, & w= \pm b^2 v^2, \quad \text { with } \delta^4=-1, \\
\ell_{\alpha, \beta, \gamma}: & \alpha a u+\beta b v+\gamma ab\sqrt{d} t=0, & w=\sqrt{2}(\alpha \beta a b u v+\beta \gamma \sqrt{d}ab^2 v t+\alpha \gamma \sqrt{d}a^2b u t), \\
\end{array}
$$
Consider the Galois extension $K:=L(a,b,\sqrt{d})$. The Galois group $\operatorname{Gal}(K/L)$ is generated by $\iota_a,\iota_b, \iota_{\sqrt{d}}$, where $\iota_{a}$ sends $a$ to $ia$ and similarly for $b$, and  $\iota_{\sqrt{d}}$ sends $\sqrt{d}$ to $-\sqrt{d}$. One has the following table:

    \begin{center} 
    \begin{tabular}{cccc}
         actions&$\iota_{a} $ &$\iota_b  $&$\iota_{\sqrt{d}} $\\
         $\ell_{t, \delta, \pm}$&$\ell_{t, i\delta, \mp}$  &$\ell_{t, -i\delta, \mp}  $& $\ell_{t, \delta, \pm}$\\
         $\ell_{u, \delta, \pm}$&$\ell_{u, -i\delta, \mp}$  &$\ell_{u, \delta, \pm}$  &$\ell_{u, -\delta, \pm}$\\
         $\ell_{v, \delta, \pm}$&$\ell_{v, \delta, \pm}$  & $\ell_{v, i\delta, \mp}$ &$\ell_{v, -\delta, \pm}$ \\
         $\ell_{\alpha, \beta, \gamma}$& $\ell_{i\alpha, \beta, i\gamma}$ &$\ell_{\alpha, i\beta, i\gamma}$  & $\ell_{\alpha, \beta, -\gamma}$\\
    \end{tabular}
    \end{center}

    In particular one can  see that   $\iota_{\sqrt{d}}=\iota_a^2\iota_b^2$ for the action on the Picard group. We can then compute actions on the basis (\ref{basis}) of $\operatorname{Pic}\bar{Y}$ and obtain:

$$\iota_a=\left(\begin{matrix}
0 & 1 & 0 & 0 & 0 & 0 & 0 & 0 \\
-1 & 0 & 0 & 0 & 0 & 0 & -1 & -1 \\
0 & 0 & 1 & 0 & 0 & 0 & 0 & 0 \\
0 & 0 & 0 & 1 & 0 & 0 & 0 & 0 \\
0 & 0 & 0 & 0 & 0 & -1 & -1 & -1 \\
0 & 0 & 0 & 0 & 1 & 0 & 0 & 0 \\
-1 & 0 & 0 & 0 & 0 & -1 & 0 & -1 \\
1 & 0 & 0 & 0 & 0 & 1 & 1 & 2 
\end{matrix}\right)$$
and
$$\iota_b=\left(\begin{matrix}
1 & 0 & 0 & 0 & 0 & 0 & 0 & 0 \\
0 & 1 & 0 & 0 & 0 & 0 & 0 & 0 \\
0 & 0 & 0 & -1 & 0 & 0 & -1 & -1 \\
0 & 0 & 1 & 0 & 0 & 0 & 0 & 0 \\
0 & 0 & 0 & 0 & 0 & 1 & 0 & 0 \\
0 & 0 & 0 & 0 & -1 & 0 & -1 & -1 \\
0 & 0 & 0 & -1 & -1 & 0 & 0 & -1 \\
0 & 0 & 0 & 1 & 1 & 0 & 1 & 2 
\end{matrix}\right)$$

The invariant space of $\iota_a$ in $\Q$ is 
$$W^{\iota_a}\otimes\Q=\left\{\left(\begin{matrix}0 \\ 0 \\ 1 \\ 0 \\ 0 \\ 0 \\ 0 \\ 0\end{matrix}\right),\left(\begin{matrix}0 \\ 0 \\ 0 \\ 1 \\ 0 \\ 0 \\ 0 \\ 0\end{matrix}\right),\left(\begin{matrix}-\frac{1}{2} \\ -\frac{1}{2} \\ 0 \\ 0 \\ -\frac{1}{2} \\ -\frac{1}{2} \\ 1 \\ 0\end{matrix}\right),\left(\begin{matrix}-\frac{1}{2} \\ -\frac{1}{2} \\ 0 \\ 0 \\ -\frac{1}{2} \\ -\frac{1}{2} \\ 0 \\ 1\end{matrix}\right)\right\}$$
and the invariant space of $\iota_b$ is 

$$W^{\iota_b}\otimes\Q=\left\{\left(\begin{matrix}1 \\ 0 \\ 0 \\ 0 \\ 0 \\ 0 \\ 0 \\ 0\end{matrix}\right),\left(\begin{matrix}0 \\ 1 \\ 0 \\ 0 \\ 0 \\ 0 \\ 0 \\ 0\end{matrix}\right),\left(\begin{matrix}0 \\ 0 \\ -\frac{1}{2} \\ -\frac{1}{2} \\ -\frac{1}{2} \\ -\frac{1}{2} \\ 1 \\ 0\end{matrix}\right),\left(\begin{matrix}0 \\ 0 \\ -\frac{1}{2} \\ -\frac{1}{2} \\ -\frac{1}{2} \\ -\frac{1}{2} \\ 0 \\ 1\end{matrix}\right)\right\}$$

Here, although the expression of the vectors involve fractions, we are only interested in linear combinations such that every entry in the result is an integer. 

We proceed to compute $(\operatorname{Pic}\bar{Y})^G$. First,  we calculate the intersection of $W^{\iota_a}\otimes\Q$ and $W^{\iota_b}\otimes\Q$. Denote by $\{u_i\}$ the basis of $W^{\iota_a}\otimes\Q$  and by $\{v_i\}$ the basis of $W^{\iota_b}\otimes\Q$. 
The kernel of the 8 by 8 matrix $[u_i,-v_i]$ is the space of coefficients which give linear combinations lying in $W^{\iota_a}\otimes\Q\cap W^{\iota_b}\otimes\Q$. We get 

$$\operatorname{Ker}[u_i,-v_i]=\left\{\left(\begin{matrix}-\frac{1}{2} \\ -\frac{1}{2} \\ 1 \\ 0 \\ -\frac{1}{2} \\ -\frac{1}{2} \\ 1 \\ 0\end{matrix}\right),\left(\begin{matrix}-\frac{1}{2} \\ -\frac{1}{2} \\ 0 \\ 1 \\ -\frac{1}{2} \\ -\frac{1}{2} \\ 0 \\ 1\end{matrix}\right)\right\}$$ 
which means the invariant spaces intersect at

$$(\operatorname{Pic}\bar{Y})^G\otimes \Q=W^{\iota_a}\otimes\Q\cap W^{\iota_b}\otimes\Q=\left\{\left(\begin{matrix}-\frac{1}{2} \\ -\frac{1}{2} \\ -\frac{1}{2} \\ -\frac{1}{2} \\ -\frac{1}{2} \\ -\frac{1}{2} \\ 1 \\ 0\end{matrix}\right),\left(\begin{matrix}-\frac{1}{2} \\ -\frac{1}{2} \\ -\frac{1}{2} \\ -\frac{1}{2} \\ -\frac{1}{2} \\ -\frac{1}{2} \\ 0 \\ 1\end{matrix}\right)\right\}$$

We are interested in linear combinations of these two vectors such that the result's entries are all integers. From  this we see that $-v_7+v_8$ is Galois stable:
$$
-v_7+v_8=-\left( \begin{matrix}-\frac{1}{2} \\ -\frac{1}{2} \\ -\frac{1}{2} \\ -\frac{1}{2} \\ -\frac{1}{2} \\ -\frac{1}{2} \\ 1 \\ 0\end{matrix}\right)+\left( \begin{matrix}-\frac{1}{2} \\ -\frac{1}{2} \\ -\frac{1}{2} \\ -\frac{1}{2} \\ -\frac{1}{2} \\ -\frac{1}{2} \\ 0 \\ 1\end{matrix}\right)
$$
Furthermore, $$-K_Y=-\left( \begin{matrix}-\frac{1}{2} \\ -\frac{1}{2} \\ -\frac{1}{2} \\ -\frac{1}{2} \\ -\frac{1}{2} \\ -\frac{1}{2} \\ 1 \\ 0\end{matrix}\right)+3\left( \begin{matrix}-\frac{1}{2} \\ -\frac{1}{2} \\ -\frac{1}{2} \\ -\frac{1}{2} \\ -\frac{1}{2} \\ -\frac{1}{2} \\ 0 \\ 1\end{matrix}\right)$$
Since $-v_7+v_8$ is not a multiple of $-K_Y$, we have $(\operatorname{Pic}\bar{Y})^G=\Z\mu\oplus \Z (-K_Y)$, where $\mu=-v_7+v_8$. 

Recall the following exact sequence:
\begin{equation}\label{ses}
0 \rightarrow \operatorname{Pic} Y \rightarrow(\operatorname{Pic} \bar{Y})^{G} \rightarrow \operatorname{Br} L \rightarrow \operatorname{Br} L(Y)
\end{equation}

From the beginning of the proof of Proposition \ref{ABCKer}, one has that the kernel of the map $\operatorname{Br} L \rightarrow \operatorname{Br} L(Y)$ is nontrivial. Hence, by \cite{CTKM07} Proposition 5.3, the only possible kernels for $\operatorname{Br} L \rightarrow \operatorname{Br} L(Y)$ (hence the cokernels for $\operatorname{Pic} Y \rightarrow(\operatorname{Pic} \bar{Y})^{G}$) are $\Z/2\Z$ and $\Z/2\Z\oplus \Z/2\Z$. 

Now, there are nonzero integers $\alpha$ and $\beta$ such that $\alpha\mu+\beta(-K_Y)\in\operatorname{Pic} Y$, since the cokernel of the map $\operatorname{Pic} Y \to (\operatorname{Pic} \bar{Y})^{G} $ is torsion.
Since $-K_Y\in \operatorname{Pic} Y$, we have $\alpha\mu$ is in $\operatorname{Pic} Y$. Take $\alpha$ minimal, then 
$$
\operatorname{Pic} Y=\Z\alpha\mu\oplus\Z(-K_Y)
$$ 
and the cokernel is just $\Z/\alpha\Z$. Explicitly, 

$$2\mu=-2v_7+2v_8=\left[\ell_{u, \zeta,+}\right]+\left[\ell_{u, \zeta^3,-}\right]+\left[\ell_{u, \zeta^5,+}\right]+\left[\ell_{u, \zeta^7,-}\right]$$ 
is in $\operatorname{Pic} Y$ but $\mu$ is not.

\end{proof}

\subsection{The case $ABC$ is not a square}

\begin{lem}\label{ABC}
    Let $Y$ be the del Pezzo surface over a field $L$, defined by
    $$
    Y:\; w^2=Au^4+Bv^4+Ct^4.
    $$
    If $ABC$ is not a square, $\operatorname{Br} L \rightarrow \operatorname{Br} L(Y)$ is injective.
\end{lem}

\begin{proof}
    Consider the Galois extension $K:=L(a,b,c)$. The Galois group $\operatorname{Gal}(K/L)$ is the abelian group generated by $\iota_a,\iota_b$, and $\iota_c$.
    \begin{center} 
    \begin{tabular}{cccc}
         actions&$\iota_{a} $ &$\iota_b  $&$\iota_{c} $\\
         $\ell_{t, \delta, \pm}$&$\ell_{t, i\delta, \pm}$  &$\ell_{t, -i\delta, \pm}  $& $\ell_{t, \delta, \mp}$\\
         $\ell_{u, \delta, \pm}$&$\ell_{u, \delta, \mp}$  &$\ell_{u, i\delta, \pm}$  &$\ell_{u, -i\delta, \pm}$\\
         $\ell_{v, \delta, \pm}$&$\ell_{v, -i\delta, \pm}$  & $\ell_{v, \delta, \mp}$ &$\ell_{v, i\delta, \pm}$ \\
         $\ell_{\alpha, \beta, \gamma}$& $\ell_{i\alpha, \beta, \gamma}$ &$\ell_{\alpha, i\beta, \gamma}$  & $\ell_{\alpha, \beta, i\gamma}$\\
    \end{tabular}
    \end{center}
$$\iota_a= \left(\begin{matrix}
-2 & -1 & -1 & -1 & -1 & -1 & -1 & -3 \\
-1 & -2 & -1 & -1 & -1 & -1 & -1 & -3 \\
-1 & -1 & -1 & -2 & -1 & -1 & -1 & -3 \\
-1 & -1 & 0 & -1 & -1 & -1 & 0 & -2 \\
-1 & -1 & -1 & -1 & -1 & 0 & 0 & -2 \\
-1 & -1 & -1 & -1 & -2 & -1 & -1 & -3 \\
-1 & -1 & 0 & -1 & -1 & 0 & -1 & -2 \\
3 & 3 & 2 & 3 & 3 & 2 & 2 & 7 
\end{matrix} \right) $$

$$\iota_b= \left(\begin{matrix}
-1 & 0 & -1 & -1 & -1 & -1 & 0 & -2 \\
-2 & -1 & -1 & -1 & -1 & -1 & -1 & -3 \\
-1 & -1 & -2 & -1 & -1 & -1 & -1 & -3 \\
-1 & -1 & -1 & -2 & -1 & -1 & -1 & -3 \\
-1 & -1 & -1 & -1 & -1 & -2 & -1 & -3 \\
-1 & -1 & -1 & -1 & 0 & -1 & 0 & -2 \\
-1 & 0 & -1 & -1 & 0 & -1 & -1 & -2 \\
3 & 2 & 3 & 3 & 2 & 3 & 2 & 7 
\end{matrix} \right) $$

and 
$$\iota_{c}= \left(\begin{matrix}
-1 & -2 & -1 & -1 & -1 & -1 & -1 & -3 \\
0 & -1 & -1 & -1 & -1 & -1 & 0 & -2 \\
-1 & -1 & -1 & 0 & -1 & -1 & 0 & -2 \\
-1 & -1 & -2 & -1 & -1 & -1 & -1 & -3 \\
-1 & -1 & -1 & -1 & -2 & -1 & -1 & -3 \\
-1 & -1 & -1 & -1 & -1 & -2 & -1 & -3 \\
0 & -1 & -1 & 0 & -1 & -1 & -1 & -2 \\
2 & 3 & 3 & 2 & 3 & 3 & 2 & 7 
\end{matrix} \right) $$

Calculating the kernel of $\iota_a-I_{8\times8}$ we see that the invariant space of $\iota_a$ is the 1-dimensional space generated by $-K_Y$. Since $-K_Y\in \operatorname{Pic}Y$ is fixed by the Galois action, we conclude that $(\operatorname{Pic}\bar{Y})^G=\Z(-K_Y)$. From the exact sequence $(\ref{ses})$ we see that $\operatorname{Br} L \rightarrow \operatorname{Br} L(Y)$ is injective.
\end{proof}

\printbibliography

\end{document}